\documentclass[preprint,12pt,3p]{elsarticle}
\usepackage{graphicx}
\usepackage{amssymb}
\usepackage{amsthm}
\usepackage{amsmath}
\usepackage{amsfonts}
\usepackage{subfigure}
\usepackage{lineno}
\usepackage{multirow}
\usepackage{threeparttable}
\usepackage{mathrsfs}
\usepackage{amsbsy}
\usepackage{url}
\usepackage[colorlinks]{hyperref}

\usepackage[noend]{algpseudocode}
\usepackage{algorithmicx}
\usepackage{algorithm}

\newtheorem{definition}{Definition}[section]
\newtheorem{theorem}{Theorem}[section]
\newtheorem{lemma}{Lemma}[section]

\newtheorem{remark}{Remark}[section]
\newtheorem{example}{Example}[section]

\journal{Expert Systems with Applications}
\begin{document}
\begin{frontmatter}

\title{A multilevel block building algorithm for fast modeling generalized separable systems}
\tnotetext[mytitlenote]{This work was supported by the National Natural Science Foundation of China (Grant No. 11532014).}
\author[lhd,UCAS]{Chen Chen}
\ead{chenchen@imech.ac.cn}
\author[lhd]{Changtong Luo\corref{cor1}}

\ead{luo@imech.ac.cn}
\author[lhd,UCAS]{Zonglin Jiang}
\ead{zljiang@imech.ac.cn}

\cortext[cor1]{Corresponding author}
\address[lhd]{State Key Laboratory of High Temperature Gas Dynamics, Institute of Mechanics, Chinese Academy of Sciences, Beijing 100190, China}
\address[UCAS]{School of Engineering Sciences, University of Chinese Academy of Sciences, Beijing 100049, China}

\begin{abstract}
Data-driven modeling plays an increasingly important role in different areas of engineering. For most of existing methods, such as genetic programming (GP), the convergence speed might be too slow for large scale problems with a large number of variables. It has become the bottleneck of GP for practical applications. Fortunately, in many applications, the target models are separable in some sense. In this paper, we analyze different types of separability of some real-world engineering equations and establish a mathematical model of generalized separable system (GS system). In order to get the structure of the GS system, a multilevel block building (MBB) algorithm is proposed, in which the target model is decomposed into a number of blocks, further into minimal blocks and factors. Compare to the conventional GP, MBB can make large reductions to the search space. This makes MBB capable of modeling a complex system. The minimal blocks and factors are optimized and assembled with a global optimization search engine, low dimensional simplex evolution (LDSE). An extensive study between the proposed MBB and a state-of-the-art data-driven fitting tool, Eureqa, has been presented with several man-made problems, as well as some real-world problems. Test results indicate that the proposed method is more effective and efficient under all the investigated cases.
\end{abstract}


\end{frontmatter}

\section{Introduction}
\label{Section1}
Data-driven modeling has become a powerful technique in different areas of engineering, such as industrial data analysis (\citeauthor{Luo2015-AST}, \citeyear{Luo2015-AST}; \citeauthor{Li2017}, \citeyear{Li2017}), circuits analysis and design (\citeauthor{Ceperic2014}, \citeyear{Ceperic2014}; \citeauthor{Shokouhifar2015}, \citeyear{Shokouhifar2015}; \citeauthor{Zarifi2015}, \citeyear{Zarifi2015}), signal processing (\citeauthor{Yang2005}, \citeyear{Yang2005}; \citeauthor{Volaric2017}, \citeyear{Volaric2017}), empirical modeling (\citeauthor{Gusel2011}, \citeyear{Gusel2011}; \citeauthor{Mehr2017}, \citeyear{Mehr2017}), system identification (\citeauthor{Guo2012}, \citeyear{Guo2012}; \citeauthor{Wong2008}, \citeyear{Wong2008}), etc. For a concerned data-driven modeling problem with $n$ input variables, we aim to find a performance function $f^*:{\mathbb{R}^n} \mapsto \mathbb{R}$ that best explains the relationship between input variables ${\mathbf{x}} = {\left[ {\begin{array}{*{20}{c}}
  {{x_1}}&{{x_2}}& \cdots &{{x_n}} 
\end{array}} \right]^{\text{T}}} \in {\mathbb{R}^n}$ and the target system (or constrained system) based on a given set of sample points ${\mathbf{S}} = {\left[ {\begin{array}{*{20}{c}}
  {{{\mathbf{x}}^{\left( 1 \right)}}}&{{{\mathbf{x}}^{\left( 2 \right)}}}& \cdots &{{{\mathbf{x}}^{\left( N \right)}}} 
\end{array}} \right]^{\text{T}}} \in {\mathbb{R}^{N \times n}}$. Among the existing methods, genetic programming (GP) (\citeauthor{Koza1992}, \citeyear{Koza1992}) is a classical approach. Theoretically, GP can get an optimal solution provided that the computation time is long enough. However, the computational cost of GP for a large scale problem with a large number of input variables is still very expensive. This situation may become even worse with increasing problem size and complexity (increasing number of independent variables and increasing target functional spaces).

Apart from basic GP, rising attention has been paid to two aspects. The first one focuses on evolutionary strategy, such as grammatical evolution (\citeauthor{ONeil2003}, \citeyear{ONeil2003}), parse-matrix evolution (\citeauthor{Luo2012-PME}, \citeyear{Luo2012-PME}), etc. These variants of GP can simplify the coding process. Other methods like clone selection programming (\citeauthor{Gan2009}, \citeyear{Gan2009}), artificial bee colony programming (\citeauthor{Karaboga2012}, \citeyear{Karaboga2012}), etc, are based on the idea of biological simulation process. However, these methods helps little on improving the convergence speed when solving large scale problems. The second branch exploits strategies to reduce the search space of solution. For instance, \citeauthor{McConaghy2011} (\citeyear{McConaghy2011}) has presented the first non-evolutionary algorithm, Fast Function eXtraction (FFX), which confined its search space to generalized linear space. However, the computational efficiency is gained at the sacrifice of losing the generality of the solution. More recently, \citeauthor{Worm2016} (\citeyear{Worm2016}) has proposed a deterministic machine learning algorithm, Prioritized Grammar Enumeration (PGE), in his thesis. PGE merges isomorphic chromosome presentations (equations) into a canonical form. However, debate still remains on how simplification effects the solving process (\citeauthor{Kinzett2008}, \citeyear{Kinzett2008}; \citeauthor{Kinzett2009}, \citeyear{Kinzett2009}; \citeauthor{McRee2010}, \citeyear{McRee2010}).

In many scientific or engineering problems, the target model are separable. \citeauthor{Luo2017} (\citeyear{Luo2017}) have presented a divide-and-conquer (D\&C) method for GP. The authors indicated that the solving process could be accelerated by detecting the correlation between each variable and the target function. In \citeauthor{Luo2017} (\citeyear{Luo2017}), a special method, bi-correlation test (BiCT), was proposed to divide a concerned target function into a number of sub-functions. Compared to conventional GP, D\&C method could reduce the computational effort (computational complexity) by orders of magnitude. More recently, \citeauthor{Chen2017-BBP} (\citeyear{Chen2017-BBP}) have proposed an improved version of BiCT method, block building technique, where the structure of the separable function is partitioned into blocks and factors. In \citeauthor{Chen2017-BBP} (\citeyear{Chen2017-BBP}), the optimization model is not linearly constructed by basis functions, but assembled with the factors and different binary operators ($+,-,\times$ and $\div$). However, many practical models are out of the scope of the mathematical form of the separable function. That is, these methods are hard to deal with complex problems in engineering.

In this paper, different types of separability of some practical engineering problems are analyzed. The form of the separable function introduced in \citeauthor{Luo2017} (\citeyear{Luo2017}) is further generalized, and a mathematical model of generalized separable system (GS system) is established. In order to get the structure of the GS system, a multilevel block building (MBB) algorithm is proposed, where the target model is decomposed into a number of block, further into minimal blocks and factors. Two types of variables in GS system are also identified, namely repeated variable and non-repeated variable. The new algorithm is an improved version of the basic block building method (\citeauthor{Chen2017-BBP}, \citeyear{Chen2017-BBP}). The performance of LDSE powered MBB is compared with the results of Eureqa, which is a state-of-the-art data-driven fitting tool. Numerical results show that the proposed algorithm is effective, and is able to recover all the investigated cases rapidly and reliably.

The rest of this paper is organized as follows. Section \ref{Section2} analyzed different types of separability in practical engineering. Section \ref{Section3} is devoted to establishing the mathematical model of GS system. In Section \ref{Section4} we propose a multilevel block building algorithm as well as some preliminary theorems. Section \ref{Section5} presents numerical results and discussions for the proposed algorithm. The paper is concluded in Section \ref{Section6} with remarking the future work.

\section{Observation and discussion of the types of separability}
\label{Section2}

\subsection{Observation}
\label{Section2-1}
As above mentioned, in many applications, the target model are separable. In this section, four examples of real-world problems are given as follows to illustrate several common types of separability in practical problems.

\begin{example}
\label{ex2-1}
When developing a rocket engine, it is crucial to model the internal flow of a high-speed compressible gas through the nozzle. The closed-form expression for the mass flow through a choked nozzle (\citeauthor{Anderson2006}, \citeyear{Anderson2006}) is
\begin{equation}
\label{nozzle}
\dot m = \frac{{{p_0}{A^*}}}{{\sqrt {{T_0}} }}\sqrt {\frac{\gamma }{R}{{\left( {\frac{2}{{\gamma  + 1}}} \right)}^{{{\left( {\gamma  + 1} \right)} \mathord{\left/
 {\vphantom {{\left( {\gamma  + 1} \right)} {\left( {\gamma  - 1} \right)}}} \right.
 \kern-\nulldelimiterspace} {\left( {\gamma  - 1} \right)}}}}}.
\end{equation}
In Eq. (\ref{nozzle}), the five independent variables, $p_0$, $T_0$, $A^*$, $R$ and $\gamma$ are all separable. The equation can be called a multiplicatively separable function, which can be re-expressed as follows
\begin{equation}
\label{nozzle2}
\begin{aligned}
\dot m & = f\left( {{p_0},{A^*},{T_0},R,\gamma } \right) \\
& = {\varphi _1}\left( {{p_0}} \right) \times {\varphi _2}\left( {{A^*}} \right) \times {\varphi _3}\left( {{T_0}} \right) \times {\varphi _4}\left( R \right) \times {\varphi _5}\left( \gamma  \right).
\end{aligned}
\end{equation}
\end{example}

The sub-functions of the multiplicatively separable function are connected with the binary operator times ($\times$). Similarly, an additively separable function whose the sub-functions are combined with the binary operator plus ($+$) can also be given. The above two classes of separable functions are very common in engineering. Furthermore, the binary operator between two sub-functions could be plus ($+$) or times ($\times$).

\begin{example}
\label{ex2-2}
In aircraft design, the lift coefficient of a whole aircraft (\citeauthor{Zhang2002}, \citeyear{Zhang2002}) can be expressed as
\begin{equation}
\label{CL}
{C_L} = {C_{L\alpha }}\left( {\alpha  - {\alpha_0}} \right) + {C_{L{\delta _e}}}{\delta _e}\frac{{{S_{{\text{HT}}}}}}{{{S_{{\text{ref}}}}}},
\end{equation}
where the variable $C_{L\alpha }$, $C_{L{\delta _e}}$, $\delta _e$, $S_{{\text{HT}}}$ and $S_{{\text{ref}}}$ are separable. The variable $\alpha$ and ${\alpha_0}$ are not separable, but their combination $\left( \alpha,\alpha_0 \right)$ can be considered separable. Hence, Eq. (\ref{CL}) can be re-expressed as
\begin{equation}
\label{CL2}
\begin{aligned}
{C_L} & =  f\left( {{C_{L\alpha }},\alpha ,{\alpha _0},{C_{L{\delta _e}}},{\delta _e},{S_{{\text{HT}}}},{S_{{\text{ref}}}}} \right) \\
& = {\varphi _1}\left( {{C_{L\alpha }}} \right) \times {\varphi _2}\left( {\alpha ,{\alpha _0}} \right) + {\varphi _3}\left( {{C_{L{\delta _e}}}} \right) \times {\varphi _4}\left( {{\delta _e}} \right)\times{\varphi _5}\left( {{S_{{\text{HT}}}}} \right) \times {\varphi _6}\left( {{S_{{\text{ref}}}}} \right).
\end{aligned}
\end{equation}
\end{example}

\begin{example}
\label{ex2-3}
The flow past a circular cylinder is a classical problem in fluid dynamics. A valid stream function for the inviscid, incompressible flow over a circular cylinder of radius $R$ (\citeauthor{Anderson2011}, \citeyear{Anderson2011}) is
\begin{equation}
\label{flowequ}
\psi = \left( {{V_\infty }r\sin \theta } \right)\left( {1 - \frac{{{R^2}}}{{{r^2}}}} \right) + \frac{\Gamma }{{2\pi }}\ln \frac{r}{R},
\end{equation}
which can be re-expressed as
\begin{equation}
\label{flowequ2}
\begin{aligned}
  \psi  &=   f\left( {{V_\infty },\sin \theta ,R,r,\Gamma } \right) \\ 
   &=   {\varphi _1}\left( {{V_\infty }} \right) \times {\varphi _2}\left( {\sin \theta } \right) \times \boxed{{\varphi _3}\left( {r,R} \right)} + {\varphi _4}\left( \Gamma  \right) \times \boxed{{\varphi _5}\left( {r,R} \right)}. 
\end{aligned}
\end{equation}
Note that the variable $r$ and $R$ appear twice in Eq. (\ref{flowequ}). In other words, variable $r$ and $R$ have two sub-functions, namely ${\varphi _3}\left( {r,R} \right) = \left( {1 - {{{R^2}} \mathord{\left/
 {\vphantom {{{R^2}} {{r^2}}}} \right.
 \kern-\nulldelimiterspace} {{r^2}}}} \right) \cdot r$ and ${\varphi _5}\left( {r,R} \right) = \ln \left( {{r \mathord{\left/
 {\vphantom {r R}} \right.
 \kern-\nulldelimiterspace} R}} \right)$. Although Eq. (\ref{flowequ2}) is not a strictly separable function, the variables of Eq. (\ref{flowequ2}) also have separability.
\end{example}

\begin{example}
\label{ex2-4}
In aircraft design, the maximum lift coefficient of a whole aircraft (\citeauthor{Zhang2002}, \citeyear{Zhang2002}) can be approximately modeled as
\begin{equation}
\label{CLm}
\left\{ \begin{aligned}
  {C_{L\max }} =  & {C_{L\max ,{\text{W}}}} - {C_{L\alpha ,WF}} \cdot \Delta {\alpha _{{\text{W/c}}}} + {C_{L\alpha ,{\text{H}}}}\left( {\frac{{{S_{\text{H}}}}}{S}} \right)\left[ {{\alpha _{CL\max }}\left( {1 - \frac{{\partial \varepsilon }}{{\partial \alpha }}} \right) - {\varepsilon _{0,{\text{H}}}} + {\psi _{\text{H}}}} \right] \\ 
   &  + {C_{L\alpha ,{\text{c}}}}\left( {\frac{{{S_{\text{c}}}}}{S}} \right)\left[ {{\alpha _{CL\max }}\left( {1 + \frac{{\partial {\varepsilon _{\text{c}}}}}{{\partial \alpha }}} \right) + {\varepsilon _{0,{\text{c}}}} + {\psi _{\text{H}}}} \right] \\ 
  {C_{L\alpha ,WF}} =  & {K_{{\text{WF}}}} \cdot {C_{L\alpha ,W}} \\ 
  {K_{{\text{WF}}}} =  & 1 + 0.025\left( {\frac{{{d_{\text{F}}}}}{b}} \right) - 0.25{\left( {\frac{{{d_{\text{F}}}}}{b}} \right)^2} \\ 
\end{aligned}  \right..
\end{equation}
This empirical formula plays an important role in the initial configuration design of an aircraft. It has a complex structure and 19 design variables, but the separability of these variables still exist in Eq. (\ref{CLm}). Similar to Eq. (\ref{flowequ}), note that variable ${C_{L\alpha ,{\text{H}}}}$, ${C_{L\alpha ,{\text{c}}}}$, ${S_{\text{H}}}$, ${S_{\text{c}}},S$ and ${\alpha _{{C_{L\max }}}}$ also appear at least twice in Eq. (\ref{CLm}). The partition of Eq. (\ref{CLm}) will be detailly analyzed in Section \ref{Section3}.
\end{example}

\subsection{Discussion}
\label{Section2-2}
As seen from the above subsection, many practical problems have the feature of separability. \citeauthor{Luo2017} (\citeyear{Luo2017}) suggested using the separability to accelerate the conventional GP for data-driven modeling. The separable function introduced in \citeauthor{Luo2017} (\citeyear{Luo2017}) could be described as follows.
\begin{definition}[Partially separable function]
\label{def1}
A scalar function $f\left( X \right)$ with $n$ continuous variables $X = \left\{ {{x_i}:i = 1,2, \cdots ,n} \right\}$ ($f:{\mathbb{R}^n} \mapsto \mathbb{R}$, $X \subset \Omega \in {\mathbb{R}^n}$, where $\Omega$ is a closed bounded convex set, such that $\Omega  = \left[ {{a_1},{b_1}} \right] \times \left[ {{a_2},{b_2}} \right] \times  \cdots  \times \left[ {{a_n},{b_n}} \right]$) is said to be separable if and only if it can be written as
\begin{equation}
\label{def1eq}
f\left( X \right) = {c_0}{ \otimes _1}{c_1}{\varphi _1}\left( {{X_1}} \right){ \otimes _2}{c_2}{\varphi _2}\left( {{X_2}} \right){ \otimes _3} \cdots { \otimes _m}{c_m}{\varphi _m}\left( {{X_m}} \right),
\end{equation}
where the variable set ${X_i}$ is a proper subset of $X$, such that ${X_i} \subset X$ with $\bigcup\nolimits_{i = 1}^m {{X_i}}  = X$, $\bigcap\nolimits_{i = 1}^m {{X_i}}  = \emptyset$, and the cardinal number of ${X_i}$ is denoted by ${\text{card}}\left( {{X_i}} \right) = {n_i}$, for $\sum\nolimits_{i = 1}^m {{n_i}}  = n$ and $i = 1,2, \cdots ,m$. 
Sub-function ${\varphi _i}$ is a scalar function such that ${\varphi _i}:{\mathbb{R}^{{n_i}}} \mapsto \mathbb{R}$. 
The binary operator $\otimes_i$ could be plus ($+$) and times ($\times$).
\end{definition}

Note that binary operator, minus ($-$) and division ($/$), are not included in $\otimes$ for simplicity. This does not affect much of its generality, since minus ($-$) could be regarded as $\left(  -  \right) = \left( { - 1} \right) \cdot \left(  +  \right)$, and sub-function could be treated as ${\tilde \varphi _i}\left(  \cdot  \right) = {1 \mathord{\left/
 {\vphantom {1 {{\varphi _i}\left(  \cdot  \right)}}} \right.
 \kern-\nulldelimiterspace} {{\varphi _i}\left(  \cdot  \right)}}$ if only ${\varphi _i}\left(  \cdot  \right) \ne 0$.
 
\begin{definition}[Completely separable function]
\label{def2}
A scalar function $f\left( {\mathbf{x}} \right)$ with $n$ continuous variables ($f:{\mathbb{R}^n} \mapsto \mathbb{R}$, $\mathbf{x} \in {\mathbb{R}^n}$) is said to be completely separable if and only if it can be rewritten as Eq. (\ref{def1eq}) and $n_i=1$ for all $i=1, 2, \cdots, m$.
\end{definition}

However, many target models in engineering (e.g., Example \ref{ex2-3} and \ref{ex2-4}) are inconsistent with the above definition of the separable function. It is because that some variables (e.g., variable ${V_\infty }$, $\sin \theta$ and $\Gamma$ of Eq. (\ref{flowequ})) appears only once in a concerned target model, while the other variables (e.g., variable $r$ and $R$ of Eq. (\ref{flowequ})) appears more than once. Examining Eq. (\ref{flowequ}) and Eq. (\ref{CLm}), it is found that the two equations share the same type of separability. This feature motivates us to generalize the mathematical form of the separable function, namely Eq. (\ref{def1eq}), and establish a more general model.

\section{The mathematical model of generalized separable system}
\label{Section3}

\subsection{Definitions}
\label{Section3-1}

In the present section, we aim to establish a more general mathematical model, generalized separable system (GS system), which could describe the type of all the typical examples discussed in Section \ref{Section2-1}. The GS system is mathematically formulated as follows.

\begin{definition}[Generalized separable system]
\label{def-GSS}
The mathematical model of a generalized separable system $f\left( X \right)$ with $n$ continuous variables $X = \left\{ {{x_i}:i = 1,2, \cdots ,n} \right\}$, ($f:{\mathbb{R}^n} \mapsto \mathbb{R}$, $X \subset \Omega \in {\mathbb{R}^n}$, where $\Omega$ is a closed bounded convex set, such that $\Omega  = \left[ {{a_1},{b_1}} \right] \times \left[ {{a_2},{b_2}} \right] \times  \cdots  \times \left[ {{a_n},{b_n}} \right]$) is defined as
\begin{equation}
\label{eq-GSS}
\begin{aligned}
f\left( X \right) &= f\left( {{X^r},{{\bar X}^r}} \right) = {c_0} + \sum\limits_{i = 1}^m {{c_i}{\varphi _i}\left( {X_i^r,\bar X_i^r} \right)}  = \sum\limits_{i = 1}^m {{c_i}{{\tilde \omega }_i}\left( {X_i^r} \right){{\tilde \psi }_i}\left( {\bar X_i^r} \right)}\\&= {c_0} + \sum\limits_{i = 1}^m {{c_i}\prod\limits_{j = 1}^{{p_i}} {{\omega _{i,j}}\left( {X_{i,j}^r} \right)} \prod\limits_{k = 1}^{{q_i}} {{\psi _{i,k}}\left( {\bar X_{i,k}^r} \right)} },
\end{aligned}
\end{equation}
where the variable set ${X^r} = \left\{ {{x_i}:i = 1,2, \cdots ,l} \right\}$ is a proper subset of $X$, such that ${X^r} \subset X$, and the cardinal number of ${X^r}$ is ${\text{card}}\left( {{{X}^r}} \right) = l$. 
${{\bar X}^r}$ is the complementary set of $X^r$ in $X$, i.e. ${{\bar X}^r} = {\complement _X}{X^r}$, where ${\text{card}}\left( {{{\bar X}^r}} \right) = n - l$. 
$X_i^r$ is the subset of ${X^r}$, such that $X_i^r \subseteq {X^r}$, where ${\text{card}}\left( {X_i^r} \right) = {r_i}$. 
$X_{i,j}^r \subseteq X_i^r$, 
such that $\bigcup\nolimits_{j = 1}^{{p_i}} {X_{i,j}^r}  = X_i^r$, $\bigcap\nolimits_{j = 1}^{{p_i}} {X_{i,j}^r}  = \emptyset$, where ${\text{card}}\left( {X_{i,j}^r} \right) = {r_{i,j}}$, for $i = 1,2, \cdots ,m$, $j = 1,2, \cdots ,{p_i}$ and $\sum\nolimits_{j = 1}^{{p_i}} {{r_{i,j}}}  = {r_i}$.
$\bar X_i^r \subset {{\bar X}^r}$ ($\bar X_i^r \ne \emptyset$), 
such that $\bigcup\nolimits_{i = 1}^m {\bar X_i^r}  = {{\bar X}^r}$, $\bigcap\nolimits_{i = 1}^m {\bar X_i^r}  = \emptyset$, where ${\text{card}}\left( {\bar X_i^r} \right) = {s_i}$, for $s_i \geqslant 1$, $\sum\nolimits_{i = 1}^m {{s_i}}  = n-l$.
$\bar X_{i,k}^r \subseteq \bar X_i^r$, 
such that $\bigcup\nolimits_{k = 1}^{{q_i}} {\bar X_{i,k}^r}  = \bar X_i^r$, $\bigcap\nolimits_{k = 1}^{{q_i}} {\bar X_i^r}  = \emptyset$, where ${\text{card}}\left( {\bar X_{i,k}^r} \right) = {s_{i,k}}$, for $k = 1,2, \cdots ,{q_i}$ and $\sum\nolimits_{k = 1}^{{q_i}} {{s_{i,k}}}  = {s_i}$. 
Sub-functions ${\varphi _i}$, ${{\tilde \omega }_i}$, ${{\tilde \psi }_i}$, ${\omega _{i,j}}$ and ${\psi _{i,k}}$ are scalar functions, 
such that ${\varphi _i}:{\mathbb{R}^{r_i+s_i}} \mapsto \mathbb{R}$, ${{\tilde \omega }_i}:{\mathbb{R}^{r_i}} \mapsto \mathbb{R}$, ${{\tilde \psi }_i}:{\mathbb{R}^{s_i}} \mapsto \mathbb{R}$, ${\omega _{i,j}}:{\mathbb{R}^{{r_{i,j}}}} \mapsto \mathbb{R}$ and ${\psi _{i,k}}:{\mathbb{R}^{{s_{i,k}}}} \mapsto \mathbb{R}$, respectively. 
${c_0},{c_1}, \cdots ,{c_m}$ are constant coefficients.
\end{definition}

\begin{definition}[Repeated variable, non-repeated variable, block and factor]
\label{def-4concepts}
In Eq. (\ref{eq-GSS}), the variables belong to ${X^r}$ and ${{\bar X}^r}$ are called repeated variables and non-repeated variables, respectively.
The sub-function ${\varphi _i}\left(  \cdot  \right)$ is called the $i$-th minimal block of $f\left( X \right)$, for $i = 1,2, \cdots ,m$.
Any combination of the minimal blocks is called a block of $f\left( X \right)$.
The sub-functions ${\omega _{i,j}}\left(  \cdot  \right)$ and ${\psi _{i,k}}\left(  \cdot  \right)$ are called the $j$-th and $k$-th factors of the repeated variables and non-repeated variables in $i$-th minimal block ${\varphi _i}\left(  \cdot  \right)$, respectively, for $j = 1,2, \cdots ,{p_i}$ and $k = 1,2, \cdots ,{q_i}$.
\end{definition}


\begin{remark}
\label{remark1}
If $m=1$, the Eq. (\ref{eq-GSS}) can be simplified as the Eq. (\ref{def1eq}), which is a special case of the GS system, referring the reader to the \citeauthor{Luo2017} (\citeyear{Luo2017}) for a more detailed analysis. In this letter, we consider $m \geqslant 2$. 
\end{remark}

\begin{remark}
\label{remark2}
If $X_i^r = \emptyset  \Leftrightarrow {r_i} = 0$, then we let ${\tilde \omega _i}\left( {X_i^r} \right) = {\omega _{i,j}}\left( {X_{i,j}^r} \right): = 1$, which indicates that there is no repeated variable in the $i$-th block ${\varphi _i}\left(  \cdot  \right)$. If ${X^r} = \emptyset$, there is no repeated variable in Eq. (\ref{eq-GSS}), then the Eq. (\ref{eq-GSS}) can also be simplified as the Eq. (\ref{def1eq}).
\end{remark}

\subsection{Interpretation}
\label{Section3-2}

In Definition \ref{def-4concepts}, two types of variables, namely repeated variable and non-repeated variable, are classified that helps us to better understand the structure of the GS system. This is the essential difference between Eq. (\ref{eq-GSS}) and Eq. (\ref{def1eq}). The non-repeated variables refer to the variables that appear only once in the GS system, while the repeated variables are the variables that appear more than once. Therefore Eq. (\ref{def1eq}) is a special case of Eq. (\ref{eq-GSS}). In Eq. (\ref{def1eq}), there is no repeated variable, and all the variables can be considered as the non-repeated variables.

The `block' and `factor' are two different types of sub-functions in the GS system. In GS system, all the blocks are connected with the binary operator plus ($+$), while the factors are connected with the binary operator times ($\times$). The minimal block is the block in which all the sub-functions (factors) are connected with times ($\times$). The factors can be considered as the `minimum element' of the GS system.

Two examples are given as follows to provide the reader with the explanations of the aforementioned concepts to fully understand the feature of such defined GS system.

\begin{example}
\label{example2-1}
Column 4, 5 and 6 of Table \ref{table2} show the repeated variables, number of minimal blocks and factors of 10 cases given in \ref{appendixA}, respectively. All the minimal blocks of the these cases are boxed in \ref{appendixA}.
\end{example}

\begin{example}
\label{example2-2}
Eq. (\ref{CLm}) is a typical GS system. Substitute ${C_{L\alpha ,WF}}$ and ${K_{{\text{WF}}}}$ into ${C_{L\max }}$, then expand ${C_{L\max }}$. The repeated variables, non-repeated variables, blocks and factors of the Eq. (\ref{CLm}) can be obtained as follows.
\begin{enumerate}[1)]
\item 6 repeated variables: ${C_{L\alpha ,{\text{H}}}}$, ${C_{L\alpha ,{\text{c}}}}$, ${S_{\text{H}}}$, ${S_{\text{c}}},S$ and ${\alpha _{{C_{L\max }}}}$.
\item 11 non-repeated variables: ${C_{L\max ,{\text{W}}}}$, ${d_{\text{F}}}$, $b$, ${C_{L\alpha ,W}}$, $\Delta {\alpha _{{\text{W/c}}}}$, $\partial \varepsilon /\partial \alpha $, $\partial {\varepsilon _{\text{c}}}/\partial \alpha $, ${\varepsilon _{0,{\text{H}}}}$, ${\varepsilon _{0,{\text{c}}}}$, ${\psi _{\text{H}}}$ and ${\psi _{\text{c}}}$. Note that the terms ${\frac{{\partial \varepsilon }}{{\partial \alpha }}}$ and ${\frac{{\partial {\varepsilon _{\text{c}}}}}{{\partial \alpha }}}$ can be considered as input variable here.


\item 8 minimal blocks: ${\varphi _1}\left( {{C_{L\max ,{\text{W}}}}} \right)$, ${\varphi _2}\left( {{d_{\text{F}}},b,{C_{L\alpha ,W}},\Delta {\alpha _{{\text{W/c}}}}} \right)$, ${\varphi _3}\left( {{C_{L\alpha ,{\text{H}}}},{S_{\text{H}}},S,{\alpha _{{C_{L\max }}}},\frac{{\partial \varepsilon }}{{\partial \alpha }}} \right)$, ${\varphi _4}\left( {{C_{L\alpha ,{\text{H}}}},{S_{\text{H}}},S,{\varepsilon _{0,{\text{H}}}}} \right)$, ${\varphi _5}\left( {{C_{L\alpha ,{\text{H}}}},{S_{\text{H}}},S,{\psi _{\text{H}}}} \right)$, ${\varphi _6}\left( {{C_{L\alpha ,{\text{c}}}},{S_{\text{c}}},S,{\alpha _{{C_{L\max }}}},\frac{{\partial {\varepsilon _{\text{c}}}}}{{\partial \alpha }}} \right)$, ${\varphi _7}\left( {{C_{L\alpha ,{\text{c}}}},{S_{\text{c}}},S,{\varepsilon _{0,{\text{c}}}}} \right)$, ${\varphi _8}\left( {{C_{L\alpha ,{\text{c}}}},{S_{\text{c}}},S,{\psi _{\text{c}}}} \right)$.

\item $\binom{8}{1}+\binom{8}{2}+\cdots+\binom{8}{8}=2^8-1$ blocks: ${{\tilde \varphi }_1}\left( {{C_{L\alpha ,{\text{H}}}},{C_{L\alpha ,{\text{c}}}},S,{S_{\text{H}}},{S_{\text{c}}},{\alpha _{{C_{L\max }}}},\frac{{\partial \varepsilon }}{{\partial \alpha }},\frac{{\partial {\varepsilon _{\text{c}}}}}{{\partial \alpha }},{\varepsilon _{0,{\text{H}}}},{\varepsilon _{0,{\text{c}}}},{\psi _{\text{H}}},{\psi _{\text{c}}}} \right)$, ${{\tilde \varphi }_2}\left( {{C_{L\max ,{\text{W}}}},{d_{\text{F}}},b,{C_{L\alpha ,W}},\Delta {\alpha _{{\text{W/c}}}}} \right)$, ${{\tilde \varphi }_3}\left( {{C_{L\max ,{\text{W}}}}} \right)$, etc.

\item 20 factors of repeated variables: ${\omega _{31}}\left( {{C_{L\alpha ,{\text{H}}}}} \right)$, ${\omega _{32}}\left( {{S_{\text{H}}}} \right){\omega _{33}}\left( S \right)$, ${\omega _{34}}\left( {\alpha _{{C_{L\max }}}} \right)$, ${\omega _{41}}\left( {{C_{L\alpha ,{\text{H}}}}} \right)$, ${\omega _{42}}\left( {{S_{\text{H}}}} \right)$, ${\omega _{43}}\left( S \right)$, ${\omega _{51}}\left( {{C_{L\alpha ,{\text{H}}}}} \right)$, ${\omega _{52}}\left( {{S_{\text{H}}}} \right)$, ${\omega _{53}}\left( S \right)$, ${\omega _{61}}\left( {{C_{L\alpha ,{\text{c}}}}} \right)$, ${\omega _{62}}\left( {{S_{\text{c}}}} \right)$, ${\omega _{63}}\left( S \right)$, ${\omega _{64}}\left( {\alpha _{{C_{L\max }}}} \right)$, ${\omega _{71}}\left( {{C_{L\alpha ,{\text{c}}}}} \right)$, ${\omega _{72}}\left( {{S_{\text{c}}}} \right)$, ${\omega _{73}}\left( S \right)$, ${\omega _{81}}\left( {{C_{L\alpha ,{\text{c}}}}} \right)$, ${\omega _{82}}\left( {{S_{\text{c}}}} \right)$ and ${\omega _{83}}\left( S \right)$.
Note that ${\omega _{1, \cdot }}\left(  \cdot  \right) = {\omega _{2, \cdot }}\left(  \cdot  \right): = 1$ in this case.
\item 10 factors of non-repeated variables: ${\psi _{11}}\left( {{C_{L\max ,{\text{W}}}}} \right)$, ${\psi _{21}}\left( {{d_{\text{F}}},b} \right)$, ${\psi _{22}}\left( {{C_{L\alpha ,W}}} \right)$, ${\psi _{23}}\left( {\Delta {\alpha _{{\text{W/c}}}}} \right)$, ${\psi _{31}}\left( {\frac{{\partial \varepsilon }}{{\partial \alpha }}} \right)$, ${\psi _{41}}\left( {{\varepsilon _{0,{\text{H}}}}} \right)$, ${\psi _{51}}\left( {{\psi _{\text{H}}}} \right)$, ${\psi _{61}}\left( {\frac{{\partial {\varepsilon _{\text{c}}}}}{{\partial \alpha }}} \right)$, ${\psi _{71}}\left( {{\varepsilon _{0,{\text{c}}}}} \right)$ and ${\psi _{81}}\left( {{\psi _{\text{c}}}} \right)$.
\end{enumerate}
\end{example}

\section{Multilevel block building}
\label{Section4}

\subsection{Minimal block detection}
\label{Section4-1}

In order to detect the separability of the GS system $f\left( X \right)$, we aim to divide $f\left( X \right)$ into a suitable number of minimal blocks, and further into factors as the typical Example \ref{example2-2}. This technique can be considered as a generalized bi-correlation test (BiCT) method. The BiCT (\citeauthor{Luo2017}, \citeyear{Luo2017}) is developed to detect the separability of a certain additively or multiplicatively separable target function, i.e. the Eq. (\ref{def1eq}). However, due to the presence of repeated variables in the GS system $f\left( X \right)$, the blocks of $f\left( X \right)$ are not additively separable. Under such circumstance, BiCT method can not be directly utilized. Before introducing minimal block detection method, a lemma and a theorem are given as follows.

\begin{lemma}
\label{lemma1}
Suppose that a GS system $f\left( X \right)$ could be initially written as $m$ additively separable blocks ($m \geqslant 1$). The test variable set ${X^t}$ is a subset of the repeated variable set ${X^r}$ of $f\left( X \right)$, namely ${X^t } \subseteq {X^r}$, if the following two conditions are satisfied.
\begin{enumerate}[1)]
\item When the variables belong to $X^t$ are all fixed, $f\left( X \right)$ can be rewritten as more than $m$ additively separable blocks.
\item When the variables belong to ${{\tilde X}^t }$, in which ${{\tilde X}^t }\ne {X^t }$ is the element of the power set $ \mathscr{P}{X^t }$ of ${X^t }$, are all fixed, $f\left( X \right)$ can not be rewritten as more than $m$ additively separable blocks.
\end{enumerate}
\end{lemma}

\begin{proof}[Sketch of proof]
Based on the definition of the repeated variable, if the repeated variables of a certain block are all fixed, the non-repeated variables of this block are additively separable. Hence, condition 1) can be easily obtained. For condition 2), if there exist a set ${{\tilde X}^t }$ belongs to $ \mathscr{P}{X^t }$ that makes $f\left( X \right)$ be rewritten as more than $m$ additively separable sub-functions. The difference set $\mathscr{P}{X^t }\backslash{{\tilde X}^t }$ must be a certain non-repeated variable set.
\end{proof}

\begin{theorem}
\label{theorem1}
Suppose that the test variable sets $X_{\left( i \right)}^t \ne \emptyset$ for $i=1,2,\cdots,k$ satisfy the two conditions of Lemma (\ref{lemma1}) in a given $m$-blocks ($m \geqslant 2$) GS system $f\left( X \right)$. Then, the repeated variable set of $f\left( X \right)$ is ${X^r}=\bigcup\nolimits_{i = 1}^k {X_{\left( i \right)}^t}$.
\end{theorem}

\begin{proof}[Sketch of proof]
This theorem can be easily proved if we let $X^t=X^r$ in the Lemma (\ref{lemma1}).
\end{proof}

Based on the aforementioned theorem, the minimal block detection technique can be succinctly described by the following algorithm.

\begin{algorithm}[htb] 
\label{algorithm1}
\caption{Minimal block detection}
\begin{algorithmic}[1]
\State Input: Target model $f:{\mathbb{R}^n} \mapsto \mathbb{R}$; Variable set $X = \left\{ {{x_1},\cdots ,{x_n}} \right\}$; Sample data ${\mathbf{S}}\in{\mathbb{R}^{N \times n}}$
\State Let all variables be randomly sampled
\State Detect the additively separable blocks by BiCT
\State Update non-repeated variable sets of each block
\For{$i=1:n$} 
  \State Create $\binom{n}{i}$ test variable sets $X_{(j)}^t (j=1,2,\cdots,\binom{n}{i})$, which consist of all possible combinations of the set $X$ taken $i$ variables at a time
  \For{$j=1:\binom{n}{i}$}
    \State Keep the test variables belong to $X_{(j)}^t$ fixed; Re-generate ${{\mathbf{S'}}}$; Call BiCT
    \If{More blocks are detected by BiCT}
      \State Add all $ {x^t } \in X_{(j)}^t$ to $X^r$
      \State Update the non-repeated variable sets of each block
    \EndIf
  \EndFor 
\EndFor
\State Output: Repeated variable set $X^r$, and non-repeated variable sets of each minimal block $\bar X_i^r$, for $i=1,2,\cdots,m$
\end{algorithmic}
\end{algorithm}

\begin{example}
Fig. \ref{2level} illustrates an example of 4-level block search for Formula (\ref{CLm}). The Formula (\ref{CLm}) could be divided into 8 minimal blocks by 4 searches, in which the first search is conducted under no repeated variable fixed, and the later three searches are performed with the repeated variables in $X_{\left( 1 \right)}^t$, $X_{\left( 2 \right)}^t$ and $X_{\left( 3 \right)}^t$ fixed, respectively. Based on Theorem \ref{theorem1}, the repeated variable set of Formula (\ref{CLm}) can be derived by ${X^r} = X_{\left( 1 \right)}^t  \cup X_{\left( 2 \right)}^t  \cup X_{\left( 3 \right)}^t$.
\end{example}

\begin{figure}[h]
\centering
\includegraphics[width=0.98\linewidth]{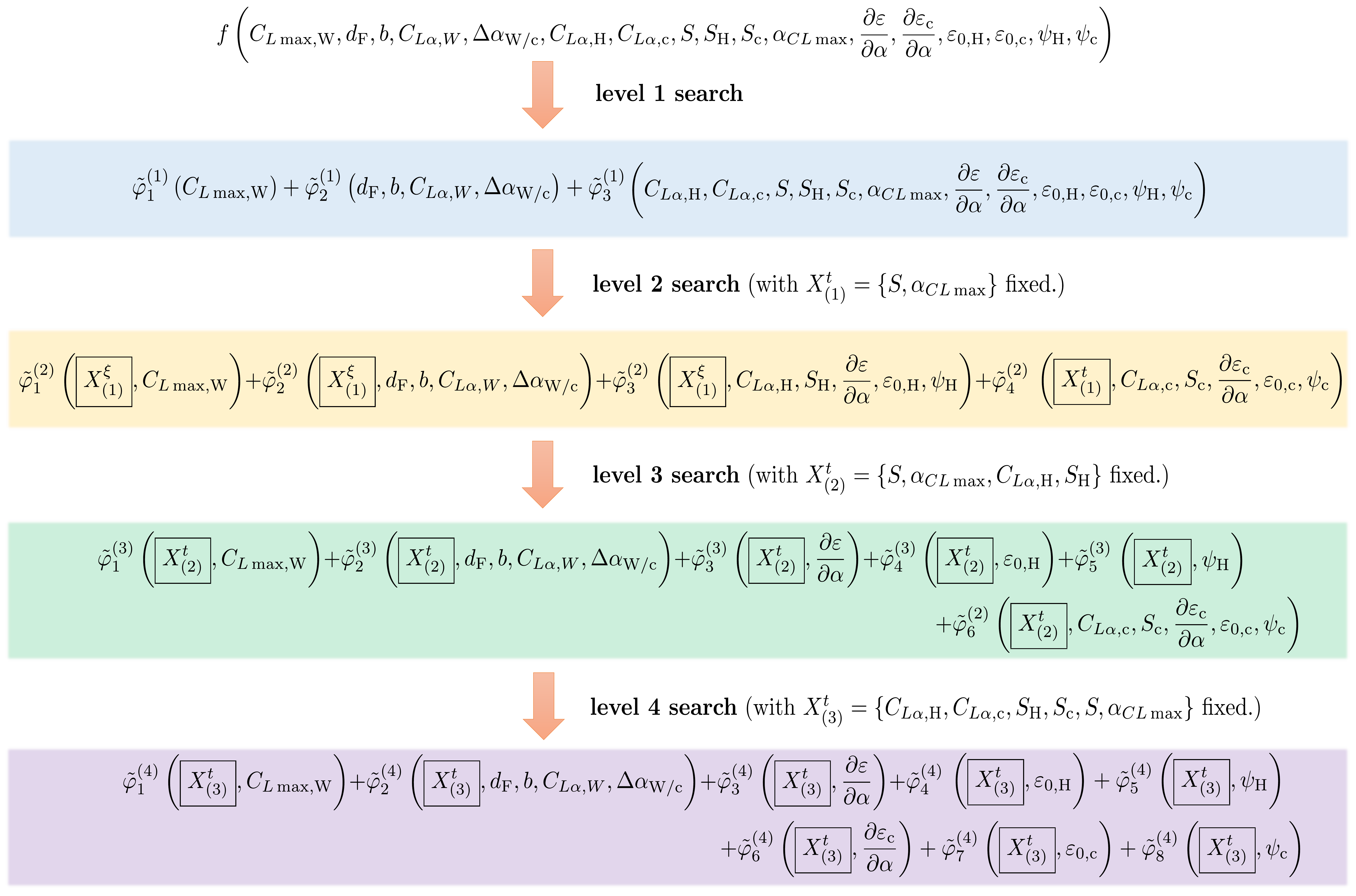}
\caption{Illustration of 4-level block search for Eq. (\ref{CLm}).}
\label{2level}
\end{figure}

\subsection{Factor detection}
\label{Section4-2}

After the sets ${X^r}$ and $\bar X_i^r$ for $i=1,2,\cdots,m$, are obtained by minimal block detection algorithm. The next step is to determine the repeated variable set of each block, namely $X_{i}^r$, and the separability of each block, namely $X_{i,j}^r$ and $\bar X_{i,k}^r$, for $j=1,2,\cdots,p_i$, $k=1,2,\cdots,q_i$. For non-repeated variables, note that in Eq. (\ref{eq-GSS}), factors ${{\psi _{i,k}}\left( {\bar X_{i,k}^r} \right)}$ are multiplicatively separable. The separability of the variables of ${\bar X_{i,k}^r}$ in the $i$-th block of Eq. (\ref{eq-GSS}) can be easily derived by BiCT method with all the variables in ${X^r}$ and $\bar X_\xi ^r\left( {\xi  \ne i} \right)$ fixed.

To determine the repeated variables of each block as well as their separability, two test functions ${{\tilde f}_1}$ and ${{\tilde f}_2}$ are formulated for BiCT. Those are
\begin{equation}
\label{sampfun1}
{\left. {f\left( X \right)} \right|_{{{\tilde x}^t } \to {x_\alpha },\bar x^r \to {x_\beta },\tilde {\bar {x}}^r \to {x_\alpha }}} \triangleq {{\tilde f}_1}\left( {{X^t }} \right)
\end{equation}
and
\begin{equation}
\label{sampfun2}
{\left. {f\left( X \right)} \right|_{{x^t } \to {x_\alpha },\bar x^r \to {x_\beta }, \tilde {\bar {x}}^r \to {x_\alpha }}} \triangleq {{\tilde f}_2}\left( {{{\tilde X}^t }} \right),
\end{equation}
where $x_1^t,x_2^t, \cdots$ are called test variables, which belong to the test variable set ${X^t }$, such that ${x^t } \in {X^t } \subset {X^r}$. Variables $\tilde x_1^t,\tilde x_2^t, \cdots$ belong to ${{\tilde X}^t}$, where ${{\tilde X}^t}$ is the complementary set of the test variable set $X^t$ in the repeated variable set $X^r$, such that ${{\tilde x}^t} \in {{\tilde X}^t} = {\complement _{{X^r}}}{X^t}$. $\bar x_1^r,\bar x_2^r, \cdots$ are the non-repeated variables of the $i$-th block, such that $\bar x^r \subset \bar X_i^r$. $\tilde {\bar {x}}_1^r,\tilde {\bar {x}}_2^r,\cdots$ are the non-repeated variables of all the blocks except the $i$-th block, such that $\tilde {\bar {x}}^r \subset {\complement _{{{\bar X}^r}}}\bar X_i^r$. ${x_\alpha }$ and ${x_\beta }$ are two certain points which will keep fixed in BiCT.

In order to eliminate the effects of other blocks when determining the separability of the repeated variables in the $i$-th block, 4 groups involving 8 times samplings should be performed to Eq. (\ref{sampfun1}) and (\ref{sampfun2}), respectively. The first sampling group are conducted as follows. Set $x_\alpha$ to the point ${x_{{F}}}$, and set $x_\beta$ to ${x_{{A_1}}}$ and ${x_{{A_2}}}$, respectively. Let variables ${{\bar x}_\xi }$ be randomly sampled. Then, two function-value vectors ${{\mathbf{f}}^{{A_1}}}$ and ${{\mathbf{f}}^{{A_2}}}$ are obtained. Let ${{\mathbf{f}}^A} = {{\mathbf{f}}^{{A_1}}} - {{\mathbf{f}}^{{A_2}}}$. The rest 3 groups of sampling are simply interpreted as follows.
\[{{\mathbf{f}}^\xi } = {{\mathbf{f}}^{{\xi _1}}} - {{\mathbf{f}}^{{\xi _2}}}, \quad \text{where}\ {x_\alpha } \to {x_F},{x_\beta } \to {x_{{\xi _1}}},{x_{{\xi _2}}},\quad \text{for}\ \xi \to B,C,D\]

\begin{theorem}
\label{theorem2}
The test variables $x_1^t,x_2^t, \cdots  \in {X^t}\subset {X^r}$ are separable, and the factor ${\omega _{i, \cdot }}\left( {{X^t}} \right)$ involves in the $i$-th block if and only the following two conditions are satisfied.
\begin{enumerate}[1)]
\item Vectors ${{\mathbf{f}}^A}$ and ${{\mathbf{f}}^B}$ are linear independent, meanwhile ${{\mathbf{f}}^C}$ and ${{\mathbf{f}}^D}$ are linear independent.
\item Both ${{\mathbf{f}}^A}$ and ${{\mathbf{f}}^B}$ are not constant vectors.
\end{enumerate}
\end{theorem}

\begin{proof}[Sketch of proof]
This theorem could be easily obtained from Theorem 1 in (\citeauthor{Luo2017}, \citeyear{Luo2017}). For condition 2), when both ${{\mathbf{f}}^A}$ and ${{\mathbf{f}}^B}$ are constant vectors, the $i$-th block does not involves repeated variable set factor ${\omega _{i, \cdot }}\left( {{X^t}} \right)$. Recall from Remark \ref{remark2} that the factor ${\omega _{i, \cdot }}\left( {{X^t}} \right)$ should set to ${\omega _{i, \cdot }}\left( {{X^t}} \right):=1$.
\end{proof}

The procedure of the factor detection becomes clear and is described in the next algorithm.

\begin{algorithm}[h]
\label{algorithm2}
\caption{Factor detection}
\begin{algorithmic}[1]
\State Input: Target system $f:{\mathbb{R}^n} \mapsto \mathbb{R}$; Sample data ${\mathbf{S}} \in {\mathbb{R}^{N \times n}}$; Set $X^r$ and $\bar X_i^r$, for $i=1,2,\cdots,m$
\For{$i=1:m$} 
  \State Let $s_i$ be the number of variables of ${\bar X_{i}^r}$, namely ${s_i}={\text{card}}\left( {\bar X_i^r} \right)$
  \For{$j=1:s_i$} 
    \State Create $\binom{s_i}{j}$ test variable set ${X_{(k)}^t}$, $k=1,2,\cdots,\binom{s_i}{j}$, which consist of all possible combinations of the set ${\bar X_{i}^r}$ taken $j$ variables at a time; Let $\text{Count}1=1$
    \For{$k=1:\binom{s_i}{j}$}
      \State Keep the test variable in ${X_{(k)}^t}$ and $X^r$ fixed; Re-generate ${{\mathbf{S'}}}$; Call BiCT
      \If{${X_{(k)}^t}$ is separable detected by BiCT}
        \State ${\bar X_{i,\text{Count}1}^r}:={X_{(k)}^t}$; $\text{Count}1:=\text{Count}1+1$
      \EndIf
    \EndFor
  \EndFor
  \State Let $r_i$ be the number of variables of ${X_{i}^r}$ (i.e., $X^r$), namely ${r_i}={\text{card}}\left( { X_i^r} \right)$
  \For{$j=1:r_i$} 
    \State Create $\binom{r_i}{j}$ test variable set ${X_{(k)}^t}$, $k=1,2,\cdots,\binom{r_i}{j}$, which consist of all possible combinations of the set ${X_{i}^r}$ taken $j$ variables at a time; Let $\text{Count}2=1$
    \For{$k=1:\binom{r_i}{j}$}
      \State Generate test functions ${{\tilde f}_1}$ and ${{\tilde f}_2}$; Obtain ${{\mathbf{f}}^A}$, ${{\mathbf{f}}^B}$, ${{\mathbf{f}}^C}$ and ${{\mathbf{f}}^D}$; Call BiCT
      \If{${X_{(k)}^t}$ satisfy Theorem \ref{theorem2}}
        \State ${ X_{i,\text{Count}2}^r}:={X_{(k)}^t}$; $\text{Count}2:=\text{Count}2+1$
      \EndIf
    \EndFor
  \EndFor
\EndFor
\State Output: Set ${X_{i,j}^r}$ and ${\bar X_{i,k}^r}$, for $j = 1,2, \cdots ,{p_i}$ and $k = 1,2, \cdots ,{q_i}$
\end{algorithmic}
\end{algorithm}

\subsection{Optimization}
\label{Section4-3}

The optimization process of MBB mainly includes two parts, namely inner optimization and outer optimization. The inner optimization will be invoked to determine the function model and coefficients of the factors ${{\omega _{i,j}}}$ and ${{\psi _{i,k}}}$. Fortunately, many state-of-the-art optimization techniques, e.g., parse-matrix evolution (\citeauthor{Luo2012-PME}, \citeyear{Luo2012-PME}), low dimensional simplex evolution (\citeauthor{Luo2012-LDSE}, \citeyear{Luo2012-LDSE}), artificial bee colony programming (\citeauthor{Karaboga2012}, \citeyear{Karaboga2012}), etc. can all be easily used to optimize the factors. Then, the optimized factors of each minimal block are multiplied together to produce minimal blocks.

The outer optimization aims at combining the minimal blocks together with the proper global parameters $c_i$. The whole process of MBB for modeling a GS system can be briefly described as follows:

\begin{description}
\item {Step 1.} (Minimal block detection) Partition a GS system into a number of minimal blocks with all the repeated variables fixed (see Algorithm \ref{algorithm1});
\item {Step 2.} (Factor detection) Divide each minimal block into factors (see Algorithm \ref{algorithm2});
\item {Step 3.} (Factor determination) Determine the factors by employing an optimization engine;
\item {Step 4.} (Global assembling) Combine the optimized factors into minimal blocks multiplicatively, further into a optimization model linearly with proper global parameters.
\end{description}

The flowchart of MBB could be briefly illustrated in Fig. \ref{workflow}.

\begin{figure}[h]
\centering
\includegraphics[width=0.98\linewidth]{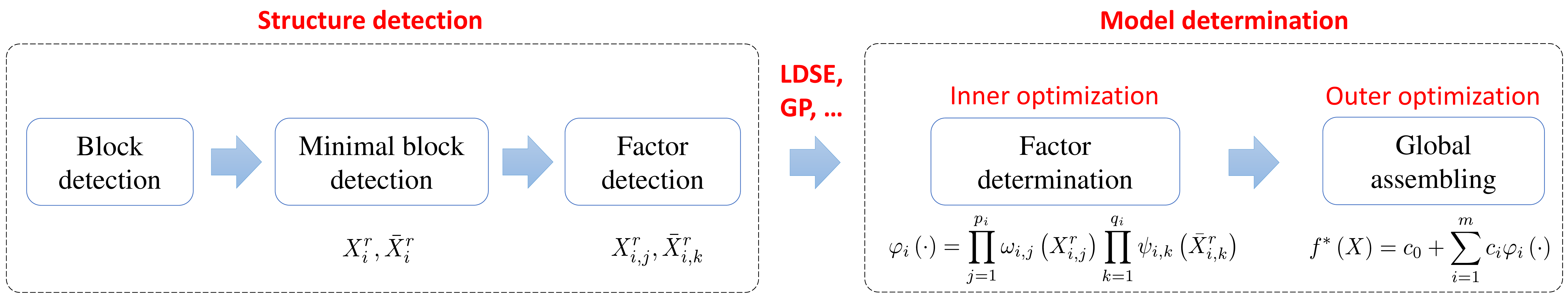}
\caption{Flowchart of MBB.}
\label{workflow}
\end{figure}

To enhance the stability and efficiency of MBB, the distribution of sample points should be as uniform as possible. Therefore, controlled sampling methods such as Latin hypercube sampling (\citeauthor{Beachkofski2002}, \citeyear{Beachkofski2002}) and orthogonal sampling (\citeauthor{Steinberg2006}, \citeyear{Steinberg2006}) are preferred for sample generation.

The proposed technique is described with functions with explicit expressions. While in practical applications, no explicit expression is available. In fact, for data-driven modeling problems, a surrogate model (\citeauthor{Forrester2009}, \citeyear{Forrester2009}) of black-box type could be established as the underlying target function in advance.

\section{Numerical results and discussion}
\label{Section5}
In our implementation, a kind of global optimization method, low dimensional simplex evolution (LDSE) (\citeauthor{Luo2012-LDSE}, \citeyear{Luo2012-LDSE}), is chosen as the optimization engine. LDSE is a hybrid evolutionary algorithm for continuous global optimization. The performances including `structure optimization' and `coefficient optimization' capabilities of LDSE powered MBB are tested by comparing with a state-of-the-art software, Eureqa (\citeauthor{Schmidt2009}, \citeyear{Schmidt2009}), which is a data-driven fitting tool based on genetic programming (GP). Eureqa was developed at the Computational Synthesis Lab at Cornell University by H. Lipson. Two test groups: 10 toy cases (see \ref{appendixA}) and 4 real-world cases (see Section \ref{Section2-1}) are taken into account.

\subsection{Toy cases}
\label{Section5-1}

In this test group, 10 problems (see \ref{appendixA}) are chosen to test MBB's performance. In Table \ref{table2}, the case number, dimension, number of sampling points, number of minimal blocks, number of factors and mean square error are denoted as Case No., Dim, No.samples, No.blocks, No.factors and MSE, respectively. 

\subsubsection{Control parameter setting}
\label{Section5-1-1}

The calculation conditions are set as follows. The number of sampling points for each independent variable is 200. The regions for cases 1-5 and 7-10 are chosen as $[-3,3]$, while case 6 is $[1,3]$. The control parameters in LDSE are set as follows. The upper and lower bounds of fitting parameters is set as $-50$ and $50$. The population size $N_p$ is set to $N_p=10+10d$, where $d$ is the dimension of the problem. The maximum generations is set to $3N_p$. Note that the maximum number of partially separable variables in all target models is 2 in our tests. Hence, our uni-variable and bi-variables function library of MBB could be set as in Table \ref{table1}, in which $m_1,m_2$ are the parameters to be fitted. Sequence search and optimization method is suitable for global optimization strategy. The search will exit immediately if the mean square error is small enough (MSE $\leqslant {\varepsilon _{{\text{target}}}}$), and the tolerance (fitting error) is ${\varepsilon _{{\text{target}}}} = {10^{ - 6}}$. In order to reduce the effect of randomness, each test case is executed 20 times.

\begin{table}[h]
\centering
\caption{Uni-variable and bi-variables preseted models in LDSE.}
\label{table1}
\begin{tabular}{lll}
\hline\hline
No. & Uni-variable model                                & Bi-variables model                                                                                                                                                                                                                             \\ \hline
1   & ${x^{{m_1}}}$           & ${m_1}{x_1} + {m_2}{x_2}$                                                                                                                                                                                            \\
2   & ${e^{{m_1}x}}$          & ${e^{{m_1}{x_1}{x_2}}}$ \\
3   & $\sin \left( {{m_1}x + {m_2}} \right)$ & ${\left( {{x_1}/{x_2}} \right)^{{m_1}}} + {m_2}{\left( {{x_1}/{x_2}} \right)^{{m_3}}}+m_4$                                                                                                                                                                                              \\
4   & $\log \left( {{m_1}x + {m_2}} \right)$           & $\sin \left( {{m_1}{x_1} + {m_2}{x_2} + {m_3}{x_1}{x_2} + {m_4}} \right)$                                                                                                                                                                     \\ \hline\hline
\end{tabular}
\end{table}

\subsubsection{Numerical results and discussion}
\label{Section5-1-2}

Table \ref{table2} shows the comparative results of MBB's and Eureqa's mean performances of the 20 independent runs with different initial populations. Numerical results show that LDSE powered MBB has successfully recovered all the target models exactly in sense of double precision. Once the uni- and bi-variables models are preseted, sequence search method makes MBB easy to find the best regression model. In practical applications, more function models could be added to the function library of MBB, provided that they are needed. Count-downing to the second column of the Table \ref{table2}, Eureqa fails to reach the target accuracy in all 20 runs within the maximum generation 100,000 for Case 6-10. Hence it is difficult for Eureqa to model large scale problems rapidly.

The computing time (CPU time) of MBB consists three parts, $t=t_1+t_2+t_3$, where $t_1$ is for the separability detection, $t_2$ for factors modeling, and $t_3$ for global assembling. In \cite{Luo2017}, authors have demonstrated that both the separability detection and function recover processes are double-precision operations and thus cost much less time than the factor determination process. That is, $t \approx t_2$. It is very easy to see that the computational efficiency of LDSE powered MBB is higher than Eureqa's. Note that LDSE powered MBB is executed on a single processor, while Eureqa is executed in parallel on 8 processors.

\begin{threeparttable}[h]
\scriptsize
\centering
\caption{Comparative results of the mean performances between LDSE powered MBB and Eureqa for modeling 10 toy cases listed in \ref{appendixA}.}
\label{table2}
\begin{tabular}{llllllll|lll}
\hline\hline
\multirow{2}{*}{\begin{tabular}[c]{@{}l@{}}Case\\ No.\end{tabular}} & \multirow{2}{*}{Dim} & \multirow{2}{*}{\begin{tabular}[c]{@{}l@{}}No.\\ samples\end{tabular}} & \multicolumn{5}{l|}{LDSE powered MBB}                                                                                                                                                                                                                                                     & \multicolumn{3}{l}{Eureqa}                                                                                                     \\ \cline{4-11} 
                                                                    &                      &                                                                        & \begin{tabular}[c]{@{}l@{}}Repeated\\ variable\end{tabular} & \begin{tabular}[c]{@{}l@{}}No.\\ block\end{tabular} & \begin{tabular}[c]{@{}l@{}}No.\\ factor\end{tabular} & \begin{tabular}[c]{@{}l@{}}CPU\\ time\tnote{1}\end{tabular} & MSE                                          & \begin{tabular}[c]{@{}l@{}}CPU\\ time\tnote{2}\end{tabular} & MSE                                             & Remarks                 \\ \hline
1                                                                   & 2                    & 400                                                                    & None                                                        & 1                                                   & 2                                                    & \textbf{7s}                                      & $\leqslant {\varepsilon _{{\text{target}}}}$ & 7s                                                 & $\leqslant {\varepsilon _{{\text{target}}}}$    & Solutions are all exact \\
2                                                                   & 3                    & 600                                                                    & None                                                        & 2                                                   & 2                                                    & \textbf{9s}                                      & $\leqslant {\varepsilon _{{\text{target}}}}$ & $>$ 4m 12s                                             & $\left[ {0,2.33} \right] \times {10^{ - 8}}$ & 10 runs failed\tnote{3}           \\
3                                                                   & 3                    & 600                                                                    & None                                                        & 2                                                   & 3                                                    & \textbf{9s}                                      & $\leqslant {\varepsilon _{{\text{target}}}}$ & $>$ 1m 9s                                              & $\leqslant {\varepsilon _{{\text{target}}}}$    & 2 runs failed           \\
4                                                                   & 3                    & 600                                                                    & $x_3$                                                       & 2                                                   & 4                                                    & \textbf{11s}                                     & $\leqslant {\varepsilon _{{\text{target}}}}$ & 55s                                                & $\leqslant {\varepsilon _{{\text{target}}}}$    & Solutions are all exact \\
5                                                                   & 4                    & 800                                                                    & $x_4$                                                       & 2                                                   & 5                                                    & \textbf{14s}                                     & $\leqslant {\varepsilon _{{\text{target}}}}$ & $>$ 2m 28s                                             & $\leqslant {\varepsilon _{{\text{target}}}}$    & 3 runs failed           \\
6                                                                   & 5                    & 1000                                                                   & $x_5$                                                       & 4                                                   & 6                                                    & \textbf{21s}                                     & $\leqslant {\varepsilon _{{\text{target}}}}$ & $\gg$ 6m 25s                                       & $\left[ {4.79,14.2} \right] \times {10^{ - 6}}$ & All runs failed         \\
7                                                                   & 5                    & 1000                                                                   & $x_4,x_5$                                                   & 3                                                   & 7                                                    & \textbf{16s}                                     & $\leqslant {\varepsilon _{{\text{target}}}}$ & $\gg$ 8m 38s                                       & $\left[ {4.05,7.68} \right] \times {10^{ - 4}}$ & All runs failed         \\
8                                                                   & 5                    & 1000                                                                   & $x_4$                                                       & 3                                                   & 6                                                    & \textbf{15s}                                     & $\leqslant {\varepsilon _{{\text{target}}}}$ & $\gg$ 6m 44s                                       & $\left[ {2.89,122.86} \right] \times {10^{ - 2}}$ & All runs failed         \\
9                                                                   & 6                    & 1200                                                                   & None                                                        & 1                                                   & 4                                                    & \textbf{9s}                                      & $\leqslant {\varepsilon _{{\text{target}}}}$ & $\gg$ 6m 59s                                       & $\left[ {1.4,8.54} \right] \times {10^{ - 1}}$ & All runs failed         \\
10                                                                  & 7                    & 1400                                                                   & $x_7$                                                       & 2                                                   & 6                                                    & \textbf{11s}                                     & $\leqslant {\varepsilon _{{\text{target}}}}$ & $\gg$ 6m 51s                                       & $\left[ {7.58,399.5} \right] \times {10^{ - 4}}$ & All runs failed         \\ \hline\hline
\end{tabular}
\begin{tablenotes}\footnotesize
\item[1] LDSE powered MBB is implemented in MATLAB and executed on a single processor.
\item[2] Eureqa is implemented in C language and executed in parallel on 8 processors.
\item[3] Maximum generation is set to 100,000.
\end{tablenotes}
\end{threeparttable}

\subsection{Real-world cases}
\label{Section5-2}

In this test group, numerical experiments on 4 real-world problems discussed in Section \ref{Section2-1} are conducted, namely Eq. (\ref{nozzle}), Eq. (\ref{CL}), Eq. (\ref{flowequ}) and Eq. (\ref{CLm}). The equations are re-written in \ref{appendixB}. This test group helps us evaluate MBB's potential capability to be applied in practical applications.

\subsubsection{Control parameter setting}
\label{Section5-2-1}
For Eq. (\ref{nozzle}), the sample set consists 100 observations uniformly distributed in a box in $R^{3}$ (i.e., $A^*=0.5:1.5$ $\text{m}^2$; $p_0=4\times 10^5:6\times 10^5$ Pa; $T_0=250:260$ K). The specific gas constant $R$ is set to $R=287$ $\text{J}/(\text{kg}\cdot \text{K})$, and the ratio of specific heats $\gamma$ is set to $\gamma=1.4$ while detecting and modeling Eq. (\ref{nozzle}).

For Eq. (\ref{CL}), the sample set consists 100 observations uniformly distributed in a box in $R^{6}$ (i.e., ${C_{L\alpha}}=0.4:0.8$; $\alpha=5:10$ degree; ${{C_{L{\delta _e}}}}=0.4:0.8$; ${{\delta _e}}=5:10$ degree; ${S_{{\text{HT}}}}=1:1.5$ $\text{m}^2$; ${{S_{{\text{ref}}}}}=5:7$ $\text{m}^2$). The zero-lift angle of attack ${{\alpha _0}}$ is $- 2^\circ$.

For Eq. (\ref{flowequ}), the sample set consists 100 observations uniformly distributed in a box in $R^{5}$ (i.e., ${{V_\infty }}=60:65$; m/s; $\theta=30:40$ degree; $r=0.2:0.5$ m; $R=0.5:0.8$ m; $\Gamma=5:10$ $\text{m}^2$/s).

For Eq. (\ref{CLm}), the sample set consists 100 observations uniformly distributed in a box in $R^{18}$ (i.e., ${C_{L\max ,{\text{W}}}}=0.4:0.8$; ${C_{L\alpha ,W}}={C_{L\alpha ,{\text{H}}}}={C_{L\alpha ,W}}=2:5$; ${{d_{\text{F}}}}=3:4$ ; $b=20:30$ $\text{m}$; ${{S_{\text{H}}}}={{S_{\text{c}}}}=1:1.5$ $\text{m}^2$; $S=5:7$ $\text{m}^2$; ${{\alpha _{CL\max }}}=10:20$ degree; ${{\varepsilon _{0,{\text{H}}}}}={{\varepsilon _{0,{\text{c}}}}}=2:4$; ${{\psi _{\text{H}}}}={{\psi _{\text{c}}}}=1:2$). In practical, the design parameter ${\partial \varepsilon /\partial \alpha }$ and ${\partial {\varepsilon _{\text{c}}}/\partial \alpha }$ should be fixed to a certain constant. However, in order to take the effect of the design parameters into consideration, we regard the two partial derivatives as two independent variables, namely ${\partial \varepsilon /\partial \alpha }={\partial {\varepsilon _{\text{c}}}/\partial \alpha }=0.5:1.5$ degree$^{-1}$.

The control parameters and the uni-variable and bi-variables model library in LDSE are set the same as the first test group. Similarly, the search will exit immediately if the mean square error is small enough (MSE $\leqslant {\varepsilon _{{\text{target}}}}={10^{ - 8}}$). In order to reduce the effect of randomness, each test case is executed 20 times, and the averaged values are reported.

\subsubsection{Numerical results and discussion}
\label{Section5-2-2}

The results of this test group help us to further analyze the performances of LDSE powered MBB. Table \ref{table3} shows MBB's results of the minimal block detections for the 4 cases. Table \ref{table4} shows the comparative results of MBB's and Eureqa's mean performances of the 20 independent runs. In Table \ref{table3}, we can see that the minimal blocks are detected correctly and fitted accurately with MSE $\leqslant {\varepsilon _{{\text{target}}}}={10^{ - 8}}$ for all 20 runs. 

Meanwhile, the comparison of MSEs in Table \ref{table4} indicates that LDSE powered MBB has much lower MSE at all cases. Note that the proposed approach enables us to decompose a large problem into many small problems, and at the same time ensures that the accuracy obtained for the local fits is carried over to the global fit.  From all of the results above, we can draw a conclusion that the LDSE powered MBB has a good capability of detecting the complex structures of GS systems, and modeling complex practical problems. Good performance for modeling complex system and the capability of `structure optimization' and `coefficient optimization' show the potential of MBB to be applied in practical applications.

\begin{threeparttable}[h]
\centering
\caption{Mean performance of LDSE powered MBB for modeling 4 real-world examples (minimal block determination) in \ref{appendixB}.}
\label{table3}
\begin{tabular}{lll}
\hline\hline
Target minimal block                                                                                            & CPU Time\tnote{1} & Result Expression                                                                                                                                                                        \\ \hline
\emph{Case 11: Eq. (\ref{nozzle})}                                                                                                     &          &                                                                                                                                                                                          \\
${\varphi _1}\left( {{p_0},{A^*},{T_0}} \right)$                                                        & 7s       & ${\varphi _1} = 0.04 * {p_0} * {A^ * } * {T_0}^{ - 0.5}$                                                                                                                                 \\ \hline
\emph{Case 12: Eq. (\ref{CL})}                                                                                                     &          &                                                                                                                                                                                          \\
${\varphi _1}\left( {{C_{L\alpha }},\alpha } \right)$                                       & 5s       & ${\varphi _1} = {C_{L\alpha }} * \left( {\alpha - 2} \right)$                                                                                                                          \\
${\varphi _2}\left( {{C_{L{\delta _e}}},{\delta _e},{S_{{\text{HT}}}},{S_{{\text{ref}}}}} \right)$      & 8s       & ${\varphi _2} = {C_{L{\delta _e}}} * {\delta _e} * {S_{{\text{HT}}}} * {1 \mathord{\left/ {\vphantom {1 {{S_{{\text{ref}}}}}}} \right. \kern-\nulldelimiterspace} {{S_{{\text{ref}}}}}}$ \\ \hline
\emph{Case 13: Eq. (\ref{flowequ})}                                                                                                     &          &                                                                                                                                                                                          \\
${\varphi _1}\left( {{V_\infty },r,\sin \theta ,R} \right)$                                             & 7s       & ${\varphi _1} = {V_\infty } * \sin \theta * \left( {r - {{{R^2}} \mathord{\left/ {\vphantom {{{R^2}} r}} \right. \kern-\nulldelimiterspace} r}} \right)$                                 \\
${\varphi _2}\left( {\Gamma ,r,R} \right)$                                                              & 5s       & ${\varphi _2} = 0.1592 * \Gamma * \ln \left( {{r \mathord{\left/ {\vphantom {r R}} \right. \kern-\nulldelimiterspace} R}} \right)$                                                       \\ \hline
\emph{Case 14: Eq. (\ref{CLm})}                                                                                                     &          &                                                                                                                                                                                          \\
${\varphi _1}\left( {{C_{L\max ,{\text{W}}}}} \right)$                                                  & 2s       & ${\varphi _1} = {C_{L\max ,{\text{W}}}}$                                                                                                                                                 \\
${\varphi _2}\left( {{d_{\text{F}}},b,{C_{L\alpha ,W}},\Delta {\alpha _{{\text{W/c}}}}} \right)$        & 5s       & ${\varphi _2} = {C_{L\alpha ,W}}*\Delta {\alpha _{{\text{W/c}}}} * \left[ {4 + 0.1 \left( {{d_{\text{F}}}/b} \right) -  {{\left( {{d_{\text{F}}}/b} \right)}^2}} \right]$      \\
${\varphi _3}\left( {{C_{L\alpha ,{\text{H}}}},{S_{\text{H}}},S,{\alpha _{CL\max }}} \right)$           & 7s       & ${\varphi _3} = {C_{L\alpha ,{\text{H}}}} * {S_{\text{H}}} * \left( {{1 \mathord{\left/
 {\vphantom {1 S}} \right.
 \kern-\nulldelimiterspace} S}} \right) * {\alpha _{CL\max }} * \left( {{{\partial \varepsilon } \mathord{\left/
 {\vphantom {{\partial \varepsilon } {\partial \alpha }}} \right.
 \kern-\nulldelimiterspace} {\partial \alpha }}} \right)$\tnote{1}                                                                                 \\
${\varphi _4}\left( {{C_{L\alpha ,{\text{H}}}},{S_{\text{H}}},S,{\varepsilon _{0,{\text{H}}}}} \right)$ & 8s       & ${\varphi _4} = {C_{L\alpha ,{\text{H}}}} * {S_{\text{H}}} * \left( {{1 \mathord{\left/
 {\vphantom {1 S}} \right.
 \kern-\nulldelimiterspace} S}} \right) * {\varepsilon _{0,{\text{H}}}}$                                                                       \\
${\varphi _5}\left( {{C_{L\alpha ,{\text{H}}}},{S_{\text{H}}},S,{\psi _{\text{H}}}} \right)$            & 8s       & ${\varphi _5} = {C_{L\alpha ,{\text{H}}}} * {S_{\text{H}}} * \left( {{1 \mathord{\left/
 {\vphantom {1 S}} \right.
 \kern-\nulldelimiterspace} S}} \right) * {\psi _{\text{H}}}$                                                                                  \\
${\varphi _6}\left( {{C_{L\alpha ,{\text{c}}}},{S_{\text{c}}},S,{\alpha _{CL\max }}} \right)$           & 10s      & ${\varphi _6} = {C_{L\alpha ,{\text{c}}}} * {S_{\text{c}}} * \left( {{1 \mathord{\left/
 {\vphantom {1 S}} \right.
 \kern-\nulldelimiterspace} S}} \right) * {\alpha _{CL\max }} * \left( {{{\partial {\varepsilon _{\text{c}}}} \mathord{\left/
 {\vphantom {{\partial {\varepsilon _{\text{c}}}} {\partial \alpha }}} \right.
 \kern-\nulldelimiterspace} {\partial \alpha }}} \right)$\tnote{2}                                                                                 \\
${\varphi _7}\left( {{C_{L\alpha ,{\text{c}}}},{S_{\text{c}}},S,{\varepsilon _{0,{\text{c}}}}} \right)$ & 8s       & ${\varphi _7} = {C_{L\alpha ,{\text{c}}}} * {S_{\text{c}}} * \left( {{1 \mathord{\left/
 {\vphantom {1 S}} \right.
 \kern-\nulldelimiterspace} S}} \right) * {\varepsilon _{0,{\text{c}}}}$                                                                       \\
${\varphi _8}\left( {{C_{L\alpha ,{\text{c}}}},{S_{\text{c}}},S,{\psi _{\text{c}}}} \right)$            & 8s       & ${\varphi _8} = {C_{L\alpha ,{\text{c}}}} * {S_{\text{c}}} * \left( {{1 \mathord{\left/
 {\vphantom {1 S}} \right.
 \kern-\nulldelimiterspace} S}} \right) * {\psi _{\text{c}}}$                                                                                  \\ \hline\hline
\end{tabular}
\begin{tablenotes}\footnotesize
\item[1] The CPU time refers to the time cost of the minimal block determination.
\item[2] The partial derivative is regard as one independent variable here.
\end{tablenotes}
\end{threeparttable}

\begin{table}[h]
\centering
\caption{Comparative results of the mean performances between LDSE powered MBB and Eureqa for modeling 4 real-world cases listed in \ref{appendixB}.}
\label{table4}
\begin{tabular}{llllll|lll}
\hline\hline
\multirow{2}{*}{\begin{tabular}[c]{@{}l@{}}Case\\ No.\end{tabular}} & \multirow{2}{*}{Dim} & \multicolumn{4}{l|}{LDSE powered MBB}                                                                                                                                                                           & \multicolumn{3}{l}{Eureqa}                                                                                        \\ \cline{3-9} 
                                                                    &                      & \begin{tabular}[c]{@{}l@{}}No.\\ block\end{tabular} & \begin{tabular}[c]{@{}l@{}}No.\\ factor\end{tabular} & \begin{tabular}[c]{@{}l@{}}CPU\\ time\end{tabular} & MSE                                          & \begin{tabular}[c]{@{}l@{}}CPU\\ time\end{tabular} & MSE                                          & Remarks       \\ \hline
11                                                                   & 3                    & 1                                                   & 3                                                    & 10s                                                & $\leqslant {\varepsilon _{{\text{target}}}}$ & 53s                                                & $\leqslant {\varepsilon _{{\text{target}}}}$ & 2 runs failed \\
12                                                                   & 6                    & 2                                                   & 6                                                    & 15s                                                & $\leqslant {\varepsilon _{{\text{target}}}}$ & $\gg$ 7m 3s                                                & $\left[ {2.04,5.63} \right] \times {10^{ - 4}}$ & All runs failed \\
13                                                                   & 5                    & 2                                                   & 6                                                    & 16s                                                & $\leqslant {\varepsilon _{{\text{target}}}}$ & $>$ 4m 41s                                                & $\left[ {0,1.96} \right] \times {10^{ - 5}}$ & 11 runs failed \\
14                                                                   & 18                   & 8                                                   & 31                                                   & 58s                                                & $\leqslant {\varepsilon _{{\text{target}}}}$ & $\gg$ 7m 12s                                                & $\left[ {3.75,11.17} \right] \times {10^{ - 1}}$ & All runs failed \\ \hline\hline
\end{tabular}
\end{table}

\section{Conclusion}
\label{Section6}
We have analyzed the different types of separability of some practical engineering problems and have established the mathematical model of the generalized separable system (GS system). In other to get the structure of the GS system, two types of variables in GS system have been identified, namely repeated variable and non-repeated variable. A new algorithm, multilevel block building (MBB), has also been proposed to decompose the GS system into a number of block, further into minimal blocks and factors. The minimal blocks and factors are optimized and assembled with a global optimization search engine, low dimensional simplex evolution (LDSE). The LDSE powered MBB is tested on several man-made test cases, as well as some real-world problems. Remarkable performance is concluded after comparing with a state-of-the-art data-driven fitting tool, Eureqa. Numerical results show the algorithm is effective, and can get the target function more rapidly and reliably. Good performance for modeling the complex problems show the potential of the GS system model and the proposed MBB algorithm to be applied in engineering. As a future work, it is planned to study on the robustness of the proposed algorithm on sample data with noises.

\section*{Acknowledgements}
This work was supported by the National Natural Science Foundation of China (Grant No. 11532014).

\appendix
\section{10 toy cases of Section \ref{Section5-1}}
\label{appendixA}
The target models which are tested in Section \ref{Section5-1} with all the minimal blocks boxed are given as follows:
\begin{description}
\item Case 1. $f\left( {\mathbf{x}} \right) = 0.5 * \boxed{{e^{{x_1}}} * \sin 2{x_2}}$, where ${x_i} \in \left[ { - 3,3} \right],i = 1,2.$

\item Case 2. $f\left( {\mathbf{x}} \right) = 2 * \boxed{\cos {x_1}} + \boxed{\sin \left( {3{x_2} - {x_3}} \right)}$, where ${x_i} \in \left[ { - 3,3} \right],i = 1,2,3.$

\item Case 3. $f\left( {\mathbf{x}} \right) = 1.2 + 10 * \boxed{\sin 2{x_1}} - 3 * \boxed{x_2^2 * \cos {x_3}}$, where ${x_i} \in \left[ { - 3,3} \right],i = 1,2,3.$

\item Case 4. $f\left( {\mathbf{x}} \right) = \boxed{{x_3} * \sin {x_1}} - 2 * \boxed{{x_3} * \cos {x_2}}$, where ${x_i} \in \left[ { - 3,3} \right],i = 1,2,3.$

\item Case 5. $f\left( {\mathbf{x}} \right) = 2 * \boxed{{x_1} * \sin {x_2} * \cos {x_4}} - 0.5 * \boxed{{x_4} * \cos {x_3}}$, where ${x_i} \in \left[ { - 3,3} \right],i = 1,2,3,4.$

\item Case 6. $f\left( {\mathbf{x}} \right) = 10 + 0.2 * \boxed{{x_1}} - 0.2 * \boxed{x_5^2 * \sin {x_2}} + \boxed{\cos {x_5} * \ln \left( {3{x_3} + 1.2} \right)} - 1.2 * \boxed{{e^{0.5{x_4}}}}$, where ${x_i} \in \left[ {1,4} \right],i = 1,2, \cdots ,5.$

\item Case 7. $f\left( {\mathbf{x}} \right) = 2 * \boxed{{x_4} * {x_5} * \sin {x_1}} - \boxed{{x_5} * {x_2}} + 0.5 * \boxed{{e^{{x_3}}} * \cos {x_4}}$, where ${x_i} \in \left[ { - 3,3} \right],i = 1,2, \cdots ,5.$

\item Case 8. $f\left( {\mathbf{x}} \right) = 1.2{\text{  +  }}2 * \boxed{{x_4} * \cos {x_2}} + 0.5 * \boxed{{e^{1.2{x_3}}} * \sin 3{x_1} * \cos {x_4}} - 2 * \boxed{\cos \left( {1.5{x_5} + 5} \right)}$, where ${x_i} \in \left[ { - 3,3} \right],i = 1,2, \cdots ,5.$

\item Case 9. $f\left( {\mathbf{x}} \right) = 0.5 * \boxed{\frac{{\cos \left( {{x_3}{x_4}} \right)}}{{{e^{{x_1}}} * {x_2}^{2}}} * \sin \left( {1.5{x_5} - 2{x_6}} \right)}$, where ${x_i} \in \left[ { - 3,3} \right],i = 1,2, \cdots ,6.$

\item Case 10. $f\left( {\mathbf{x}} \right) = 1.2 - 2 * \boxed{\frac{{{x_1} + {x_2}}}{{{x_3}}} * \cos {x_7}} + 0.5 * \boxed{{e^{{x_7}}} * {x_4} * \sin \left( {{x_5}{x_6}} \right)}$, where ${x_i} \in \left[ { - 3,3} \right],i = 1,2, \cdots ,7.$
\end{description}

\section{4 real-world cases of Section \ref{Section5-2}}
\label{appendixB}
4 real-world cases of Example \ref{ex2-1}-\ref{ex2-4}, namely Eq. (\ref{nozzle}), Eq. (\ref{CL}), Eq. (\ref{flowequ}) and Eq. (\ref{CLm}), are re-written in dimensionless forms with variable $x_i$ (which is used in our program) instead of physical variable in this appendix, respectively. All the minimal blocks of the following equations are boxed.

\begin{description}
\item Case 11. Eq. (\ref{nozzle}) is re-written as

\begin{equation}
f\left( {\mathbf{x}} \right) = 4 \times 10^3 * \boxed{\frac{{{x_1} * {x_2}}}{{\sqrt {{x_3}} }}},
\end{equation}

where ${x_1} = {p_0} \in \left[ {4,6} \right],{x_2} = {A^ * } \in \left[ {0.5,1.5} \right],{x_3} = {T_0} \in \left[ {250,260} \right]$. $\gamma$ and $R$ of Eq. (\ref{nozzle}) are set to $\gamma  = 1.4$ and $R = 287$.

\item Case 12. Eq. (\ref{CL}) is re-written as

\begin{equation}
f\left( {\mathbf{x}} \right) = \boxed{{x_1} * \left( {{x_2} - 2} \right)} + \boxed{{x_3} * {x_4} * \frac{{{x_5}}}{{{x_6}}}},
\end{equation}

where ${x_1} = {C_{L\alpha }} \in \left[ {0.4,0.8} \right],{x_2} = \alpha  \in \left[ {5,10} \right],{x_3} = {C_{L{\delta _e}}} \in \left[ {0.4,0.8} \right],{x_4} = {\delta _e} \in \left[ {5,10} \right],{x_5} = {S_{{\text{HT}}}} \in \left[ {1,1.5} \right],{x_6} = {S_{{\text{ref}}}} \in \left[ {5,7} \right]$. ${\alpha _0}$ Eq. (\ref{CL}) is set to ${\alpha _0} =  - 2$.

\item Case 13. Eq. (\ref{flowequ}) is re-written as

\begin{equation}
f\left( {\mathbf{x}} \right) = \boxed{{x_1} * {x_2} * {x_5}\left( {1 - \frac{{x_4^2}}{{x_5^2}}} \right)} + \frac{1}{{2\pi }} * \boxed{{x_3} * \ln \frac{{{x_5}}}{{{x_4}}}},
\end{equation}

where ${x_1} = {V_\infty } \in \left[ {60,65} \right],{x_2} = \theta  \in \left[ {30,40} \right],{x_3} = \Gamma  \in \left[ {5,10} \right],{x_4} = R \in \left[ {0.5,0.8} \right],{x_5} = r \in \left[ {0.2,0.5} \right]$.

\item Case 14. Eq. (\ref{CLm}) is re-written as

\begin{equation}
\begin{aligned}
  f\left( {\mathbf{x}} \right) = \boxed{{x_1}} &  - 0.25 * \boxed{{x_4} * {x_5} * {x_6} * \left[ {4 + 0.1 \left( {\frac{{{x_2}}}{{{x_3}}}} \right) - {{\left( {\frac{{{x_2}}}{{{x_3}}}} \right)}^2}} \right]} + \boxed{{x_{13}} * \frac{{{x_{14}}}}{{{x_{15}}}} * {x_{18}} * {x_7}} &  \\ 
   &  - \boxed{{x_{13}} * \frac{{{x_{14}}}}{{{x_{15}}}} * {x_8}} + \boxed{{x_{13}} * \frac{{{x_{14}}}}{{{x_{15}}}} * {x_9}} + \boxed{{x_{16}} * \frac{{{x_{17}}}}{{{x_{15}}}} * {x_{18}} * {x_{10}}}  &  \\ 
   & - \boxed{{x_{16}} * \frac{{{x_{17}}}}{{{x_{15}}}} * {x_{11}}}   + \boxed{{x_{16}} * \frac{{{x_{17}}}}{{{x_{15}}}} * {x_{12}}}, \\ 
\end{aligned}
\end{equation}

where ${x_1} = {C_{L\max ,{\text{W}}}} \in \left[ {0.4,0.8} \right],{x_{4,13,16}} = {C_{L\alpha ,W}} = {C_{L\alpha ,{\text{H}}}} = {C_{L\alpha ,{\text{c}}}} \in \left[ {2,5} \right],{x_2} = {d_{\text{F}}} \in \left[ {3,4} \right],{x_3} = b \in \left[ {20,30} \right],{x_{14,17}} = {S_{\text{H}}} = {S_{\text{c}}} \in \left[ {1,1.5} \right],{x_{15}} = S \in \left[ {5,7} \right],{x_{18}} = {\alpha _{CL\max }} \in \left[ {10,20} \right],{x_{8,11}} = {\varepsilon _{0,{\text{H}}}} = {\varepsilon _{0,{\text{c}}}} \in \left[ {1,1.5} \right],{x_{9,12}} = {\psi _{\text{H}}} = {\psi _{\text{c}}} \in \left[ {1,2} \right]$. The partial derivative $ \frac{{\partial \varepsilon }}{{\partial \alpha }}$ and $\frac{{\partial {\varepsilon _{\text{c}}}}}{{\partial \alpha }}$ of Eq. (\ref{CLm}) are set to ${x_{7,10}} = \frac{{\partial \varepsilon }}{{\partial \alpha }} = \frac{{\partial {\varepsilon _{\text{c}}}}}{{\partial \alpha }} \in \left[ {0.5,1.5} \right]$.
\end{description}

\bibliographystyle{model5-names}

\bibliography{main}

\end{document}